\documentclass[11pt,twoside]{article}
\usepackage{times}
\usepackage{amsmath,amssymb,amsthm}
\usepackage{enumerate}
\usepackage{cite}
\allowdisplaybreaks

\newcommand{\R}{\mathbb R}

\newcommand{\p}{\partial}
\newcommand{\ve}{\varepsilon}
\newcommand{\f}{\frac}

\newcommand{\la}{\lambda}

\newcommand{\al}{\alpha}

\newcommand{\vp}{\varphi}

\newcommand{\dl}{\delta}

\newcommand{\ds}{\displaystyle}
\newcommand{\md}{\mathrm{d}}

\allowdisplaybreaks

\newcommand{\crit}{\textup{crit}}
\newcommand{\conf}{\textup{conf}}

\theoremstyle{plain}
\newtheorem{theorem}{Theorem}[section]

\newtheorem{lemma}[theorem]{Lemma}

\newtheorem{remark}{Remark}[section]

\numberwithin{equation}{section}
\title{On semilinear Tricomi equations in one space dimension}
  \author{Daoyin He$^{1*}$,\qquad Ingo Witt$^{2,*}$, \qquad Huicheng
  Yin$^{3,}$\footnote{He Daoyin (\texttt{daoyin@mathematik.uni-goettingen.de})
    and Yin Huicheng
    (\texttt{huicheng@} \texttt{nju.edu.cn}) are supported by the NSFC
    (No.~11571177, No.~11731007) and by the Priority Academic Program Development of
    Jiangsu Higher Education Institutions.}\vspace{0.5cm}\\ \small 1.
  School of Mathematical Sciences, Fudan University, Shanghai
  200433, China.\\ \small 2.  Mathematical Institute, University of
  G\"{o}ttingen, Bunsenstr.~3-5, D-37073 G\"{o}ttingen,
  Germany.\\ \small 3.  School of Mathematical Sciences and Mathematical Institute, Nanjing Normal University,\\
 \small  Nanjing 210023, China.\\}
\vspace{0.5cm}
\begin{document}
\date{}

\maketitle
\thispagestyle{empty}

\begin{abstract}
For 1-D semilinear Tricomi equation $\partial_t^2 u-t\p_x^2u=|u|^p$ with initial data $(u(0,x), \p_t u(0,x))$
$=(u_0(x), u_1(x))$,
where $t\ge 0$, $x\in\R$, $p>1$, and $u_i\in
C_0^{\infty}(\R)$ ($i=0,1$),  we shall prove that there exists a critical exponent
$p_{\crit}=5$ such that the small data weak solution $u$ exists globally when
$p>p_{\crit}$; on the other hand, the weak solution $u$, in general,
blows up in finite time when $1<p<p_{\crit}$. We specially point out
that for 1-D semilinear wave equation $\partial_t^2 v-\p_x^2v=|v|^p$, the weak solution $v$
will generally blow up in finite time
for any $p>1$.
By this paper and \cite{HWYin1}-\cite{HWYin3}, we have given a systematic study
on the blowup or
global existence of small data solution $u$ to
the equation $\partial_t^2 u-t\Delta u=|u|^p$  for all
space dimensions. One of the main ingredients in the paper is to establish
a crucial weighted Strichartz-type inequality  for 1-D linear degenerate equation
$\partial_t^2 w-t\partial_x^2 w=F(t,x)$ with $(w(0,x), \p_tw(0,x))=(0,0)$,
i.e., an inequality with the weight $(\f{4}{9}t^3-|x|^2)^{\al}$
between the solution $w$ and the function $F$
is derived for some real numbers $\al$.
\end{abstract}

\vskip 0.2 true cm


\qquad\qquad \quad blowup.

\vskip 0.2 true cm


\vskip 0.3 true cm

\section{Introduction}
In our former papers \cite{HWYin1, HWYin2, HWYin3}, for the multi-dimensional
semilinear Tricomi equation $\partial_t^2 u-t\Delta u=|u|^p$ with initial data $(u(0,x), \p_t u(0,x))$
$=(u_0(x), u_1(x))$,
where $t\ge 0$, $x\in\R^n$ with $n\ge 2$, $p>1$, and $u_i\in
C_0^{\infty}(\R^n)$ ($i=0,1$), we have given a systematic study
on the blowup or
global existence of small data solution $u$.
In this paper, we focus on the 1-D semilinear Tricomi equation:
\begin{equation}\label{equ:original}
\left\{ \enspace
\begin{aligned}
&\partial_t^2 u-t\partial_x^2u=|u|^p \quad \text{in} \quad \R_+^{1+1},\\
&u(0,x)=u_0(x), \quad \partial_{t} u(0,x)=u_1(x),
\end{aligned}
\right.
\end{equation}
where $p>1$, $u_i(x)\in
C_0^{\infty}(\R)$ ($i=0,1$) and \(\operatorname{supp}u_i\in (-M, M)\) for some fixed constant \(M>1\).
For the local existence and regularity of
solution~$u$ of \eqref{equ:original} under weaker regularity
assumptions on $(u_0, u_1)$, the reader may consult
\cite{Rua1}-\cite{Rua4} and \cite{Yag2}-\cite{Yag3}.

Our present purpose is to
determine a critical exponent $p_{\crit}=5$ such that
the small data weak solution $u$  of \eqref{equ:original} exists globally when
$p>p_{\crit}$; on the other hand, the weak solution $u$, in general,
blows up in finite time when $1<p<p_{\crit}$.
Since the local existence of  weak solution $u$
to semilinear Tricomi equations with minimal regularities has been established in \cite{Rua4}, without loss of generality,
as in \cite{HWYin2, HWYin3} we only focus on the global small data  weak solution problem of \eqref{equ:original} starting from some
positive time $T_0>0$. Therefore, it is plausible that one utilizes the nonlinear function
$F_p(t,u)=\big(1-\chi(t)\big)F_p(u)+\chi(t)|u|^p$ instead of $|u|^p$ in \eqref{equ:original}, where
$F_p(u)$ is a $C^{\infty}-$smooth function with $F_p(0)=0$ and
$|F_p(u)|\leq C(1+|u|)^{p-1}|u|$, and $\chi(s)\in C^{\infty}(\Bbb R)$
with \(\chi(s)=
\left\{ \enspace
\begin{aligned}
1, \quad &s\geq T_0, \\
0, \quad &s\leq T_0/2.
\end{aligned}
\right.\)
\quad  Correspondingly, we shall study the following problem instead
of \eqref{equ:original}
\begin{equation}
\left\{ \enspace
\begin{aligned}
&\partial_t^2 u-t\partial_x^2 u =F_p(t,u)\quad \text{in} \quad \R_+^{1+1},\\
&u(0,x)=\ve u_0(x), \quad \partial_{t} u(0,x)=\ve u_1(x).\\
\end{aligned}
\right.
\label{equ:1.2}
\end{equation}
\begin{theorem}[{\bf Global existence for $p>p_{\crit}$}]\label{thm:1.2}
Assume that $p>p_{\crit}\equiv 5$. Then there exists a constant $\ve_0>0$ such that, for \(0<\ve\leq\ve_0\), problem
\eqref{equ:1.2} admits a global weak solution $u$ such that
\begin{equation}\label{equ:1.3}
\left(1+\big|\phi^2(t)-|x|^2\big|\right)^{\gamma}u\in L^{p+1}(\mathbb{R}^{1+1}_+),
\end{equation}
where  $\phi(t)=\f{2}{3}t^{\f{3}{2}}$, and the positive constant $\gamma$ fulfills
\begin{equation}\label{equ:1.4}
0<\gamma<\f{1}{6}-\f{5}{6(p+1)}.
\end{equation}
\end{theorem}

With respect to the case of  $1<p<p_{\crit}$, we have

\begin{theorem}[{\bf Blow up for $1<p<p_{\crit}$}]\label{thm1.1}
  Let $1<p<p_{\crit}\equiv 5$. In addition, $u_i\ge 0$ and
  $u_i\not\equiv 0$ for $i=0, 1$ are assumed. Then problem \eqref{equ:original} admits
  no global weak  solution $u$ with $u\in C\left([0, \infty), H^1(\R)\right)
  \cap C^1\left([0, \infty), L^2(\R)\right)$.
\end{theorem}

\begin{remark}
For brevity, in the  present paper we only study the semilinear Tricomi equation instead of the
generalized semilinear Tricomi equation $\partial_t^2 u-t^m\Delta u =|u|^p$
$(m\in\Bbb N)$ in problem \eqref{equ:original}. In fact, by the analogous methods in Theorem 1.1-
Theorem 1.2 and \cite{HWYin1}-\cite{HWYin2}, we can establish the same results to Theorem 1.1-Theorem 1.2 for the
generalized semilinear Tricomi equation $\partial_t^2 u-t^m\partial_x^2u=|u|^p$ ($m\in\Bbb N$) with the critical power $p_{\crit}(m)=1+\f{4}{m}$ .
\end{remark}

\begin{remark}
For the 1-D semilinear wave equation $\partial_t^2 v-\partial_x^2 v=|v|^p$ $(p>1)$, direct computation
shows that the local weak  solution $v$ will
generally blow up in finite time, see for example \cite{Gla3}. However, for the 1-D semilinear
Tricomi equation $\partial_t^2 u-t\partial_x^2 u =|u|^p$, by Theorem 1.1
we know that the global small data  weak solution $u$ exists  for \(p>5\),
which is established through getting the decay property of solutions to
the linear Tricomi equation (see \eqref{Y-1} below) and through deriving
some weighted Strichartz inequalities (see Theorem 2.1 and Theorem 3.2).
\end{remark}

\begin{remark}
For the 1-D linear wave equation $\partial_t^2 v-\partial_x^2 v=0$ with  $(v(0,x), \p_tv(0,x))=(\vp_0(x), 0)$,
it follows from D'Alembert's formula that $v(t,x)=\f12(\vp_0(x+t)+\vp_0(x-t))$. If $\vp_0(x)\in H^1(\R)$,
one then has that $v\in L_t^{\infty}([0, \infty), L_x^p(\R))$ for $1\le p\le\infty$ but
$v\not\in L_t^q([0, \infty), L_x^p(\R))$ for any $q$ satisfying $1\le q<\infty$, namely, there is no
global Strichartz-type inequality for the solution $v$ of 1-D wave equation.
However, it is not the case for the 1-D linear Tricomi equation (see Theorem 2.1 and Theorem 3.2).
\end{remark}

\begin{remark} In \cite{HWYin1}-\cite{HWYin3},
by the expression of the solution $w$ to the M-D linear equation $\partial_t^2 w-t\Delta w=F(t,x)$
with $(w(0,x), \p_tw(0,x))=(f, g)$, we have established
such a weighted Strichartz inequality $\left\|\big(\f{4}{9}t^3-|x|^2\big)^{\gamma_1}w\right\|_{L^q(\mathbb{R}^{1+n}_+)}
\leq C\big(\parallel f\parallel_{W^{\f{n}{2}+\f{1}{3}+\dl,1}(\mathbb{R}^n)}
+\parallel g\parallel_{W^{\f{n}{2}-\f{1}{3}+\dl,1}(\mathbb{R}^n)}
+\big\|\big(\f{4}{9}t^{3}-|x|^2\big)^{\gamma_2}F\big\|_{L^{\frac{q}{q-1}}(\mathbb{R}^{1+n}_+)}\big)$
for suitable {\bf positive} numbers $\gamma_1$, $\gamma_2$, $\dl$ and $q>1$. However, for the 1-D case of $\partial_t^2 w-t\p_x^2 w=F(t,x)$
with $(w(0,x), \p_tw(0,x))=(0, 0)$, we can derive such an inequality
$\left\|\big(\f{4}{9}t^3-|x|^2\Big)^{\mu_1}w\right\|_{L^q(\mathbb{R}^{1+1}_+)}\leq C\big\|\big(\f{4}{9}t^3-|x|^2\big)^{\mu_2}F\big\|_{L^{\frac{q}{q-1}}(\mathbb{R}^{1+1}_+)}$
for $q>1$ and suitable {\bf real} numbers $\mu_1$ and $\mu_2$ (here $\mu_1$ is {\bf negative}
and $\mu_2$ may be {\bf negative}, see Theorem 3.1 below).
\end{remark}

\begin{remark}
For the M-D semilinear generalized
Tricomi equation $\partial_t^2 u-t^m \Delta u=|u|^p$ with $\big(u(0,x),$ $\p_t u(0,x)\big)=\big(u_0(x), u_1(x)\big)$
and $x\in\Bbb R^n$ $(n\ge 2)$, in \cite{HWYin1}-\cite{HWYin3} we have shown that
there exists a critical exponent $p_{crit}(m, n)>1$ such that the  weak solution $u$ generally blows up
when $1 < p<p_{crit}(m, n)$ and meanwhile there exists a global small data weak  solution $u$
when $p>p_{crit}(m, n)$, where $p_{crit}(m,n)$ is the positive root of the algebraic equation
\begin{equation}\label{W-1}
((m+2)\frac{n}{2}-1)p^2+((m+2)(1-\frac{n}{2})-3)p-(m+2)=0.
\end{equation}
If we formally let $m=1$ and $n=1$ in \eqref{W-1} (actually only holds for $n\ge 2$), then
$p_{crit}(1, 1)=\f{3+\sqrt{33}}{2}$. It is easy to know $p_{crit}(1, 1)<5$,
which means that $p_{crit}(1, 1)$ is strictly less than $p_{crit}=5$ in
Theorem 1.1 and Theorem 1.2.
\end{remark}

The linear equation $\p_t^2u-t\p_x^2u=0$ is the well-known Tricomi equation
which arises from transonic gas dynamics (see \cite{Bers} and \cite{Mora}).
There are extensive results
for both linear and semilinear Tricomi equations in $n$ space dimensions
$(n\in\Bbb N)$. For
instances, with respect to the linear Tricomi equation  $\p_t^2u-t\Delta u=0$,
the authors in \cite{Bar}, \cite{Yag1} and \cite{Yag3} have computed its
fundamental solution explicitly; with respect to the semilinear Tricomi equation $\p_t^2u-t\Delta u=f(t,x,u)$,
under some certain assumptions on the function
$f(t,x,u)$, the authors in \cite{Gva} and \cite{Lup2}-\cite{Lup4}
have obtained a series of interesting  results on the existence and uniqueness of solution $u$
in bounded domains; with respect to the  Cauchy problem of semilinear Tricomi equations,
the authors in \cite{Beals} and \cite{Rua1, Rua3, Rua4} established the local existence as well
as the singularity structure of low regularity solutions in the
degenerate hyperbolic region and the elliptic-hyperbolic mixed region,
respectively. In addition, we have given a complete study
on the blowup or
global existence of small data solution $u$ to
the  semilinear Tricomi equation $\partial_t^2 u-t\Delta u=|u|^p$  for $n\ge 2$
(see \cite{HWYin1}-\cite{HWYin3}). In the present paper, we shall systematically
study the 1-D semilinear Tricomi equation $\partial_t^2 u-t\p_x^2 u=|u|^p$.

We now comment on the proof of  Theorem~\ref{thm:1.2} and Theorem~\ref{thm1.1}.
To prove the global existence in
Theorem~\ref{thm:1.2}, we require to establish some weighted Strichartz estimates for the
Tricomi operator $\p_t^2-t\partial_x^2$ as in \cite{HWYin2}.
In this process, a series of
inequalities are derived by applying an explicit formula for the solution $v$
of linear Tricomi equation $\p_t^2v-t\partial_x^2v=f(t,x)$ and by utilizing a basic observation from \cite{Gls} together with
some delicate analysis. Here we point out that since the stationary phase method for treating the
M-D problem in
\cite{HWYin2} and \cite{Ls} is not applicable for the 1-D case, we can not get a suitable \(L^1-L^\infty\) estimate of $v$ and
then use interpolation between  \(L^1-L^\infty\) estimate and  \(L^2-L^2\) estimate to get the Strichartz-type estimate of $v$
as in \cite{HWYin2}.
Based on the resulting Strichartz inequalities and the contraction mapping principle,
we complete the proof of Theorem~\ref{thm:1.2}.
To prove Theorem~\ref{thm1.1}, we define the
function $G(t)=\int_{\R}u(t,x)\ \md x$. By
an analysis similar to \cite{HWYin1}, and motivated by \cite{Yor},  we can derive a
Riccati-type ordinary differential inequality for $G(t)$ through a delicate
analysis of \eqref{equ:original}. From this and Lemma 2.1 in \cite{Yor},
the blowup result for \(1<p<p_{crit}\) in Theorem 1.2 is established under the positivity assumptions of $u_0(x)$ and $u_1(x)$.

This paper is organized as follows: In Section 2,
some weighted Strichartz estimates  for the linear homogeneous Tricomi equation
are established. In Section 3, for the linear inhomogeneous Tricomi equation,
the related weighted Strichartz estimates are derived. By applying the results in Section 2 and Setion 3, Theorem~\ref{thm:1.2} is proved
in Section 4. In Section 5,
we complete the proof of Theorem 1.2.

\section{Mixed-norm estimate for homogeneous equation}

In order to establish  the global existence of  weak solution
$u$ to problem \eqref{equ:original},
we shall derive some mixed space-time norm estimates for the corresponding
linear problem.

At first, we consider the following homogeneous problem
\begin{equation}
\begin{cases}
&\partial_t^2 v-t\partial_x^2 v=0 \qquad \text{in} \quad \R_+^{1+1},\\
&v(0,x)=f(x),\quad \partial_tv(0,x)=g(x), \\
\end{cases}
\label{equ:3.1}
\end{equation}
where $f, g\in C_0^\infty(\mathbb{R})$,  $\operatorname{supp}(f,g)\subseteq \{x: |x|\leq M\}$ for some fixed constant
$M>1$. We now derive a weighted space-time estimate of Strichartz-type for the solution $v$.
\begin{theorem}\label{thm:3.1}
For the solution $v$ of \eqref{equ:3.1}, one then has
\begin{equation}\label{equ:3.2}
\left\|\Big(\big(\phi(t)+M\big)^2-|x|^2\Big)^\gamma v\right\|_{L^q(\mathbb{R}^{1+1}_+)}\leq C(\parallel f\parallel_{W^{\f{5}{6}+\delta,1}(\mathbb{R})}+\parallel g\parallel_{W^{\f{1}{6}+\delta,1}(\mathbb{R})}),
\end{equation}
where $\phi(t)=\frac{2}{3}t^{\frac{3}{2}}$, \(q=1+p\), \(p>p_{\crit}\),  $\gamma<\f{1}{6}-\f{5}{6q}$,
$0<\delta<\f{1}{6}-\gamma-\f{5}{6q}$,
and $C$ is a positive constant depending only on $q$, $\gamma$ and $\delta$.
\end{theorem}

\begin{proof}
It follows from \cite{Yag2} that the solution $v$ of \eqref{equ:3.1} can be expressed as
\[v(t,x)=V_1(t, D_x)f(x)+V_2(t, D_x)g(x),\]
where the symbols $V_j(t, \xi)$ ($j=1,2$) of the Fourier integral operators $V_j(t, D_x)$  are
\begin{equation}\label{equ:3.3}
\begin{split}
V_1(t,|\xi|)=&\frac{\Gamma(\frac{1}{3})}{\Gamma(\f{1}{6})}\biggl[e^{\frac{z}{2}}H_+\Big(\f{1}{6},\frac{1}{3};z\Big) +e^{-\frac{z}{2}}H_-\Big(\f{1}{6},\frac{1}{3};z\Big)\biggr]
\end{split}
\end{equation}
and
\begin{equation}
\begin{split}
V_2(t,|\xi|)=&\frac{\Gamma(\frac{5}{3})}{\Gamma(\frac{5}{6})}t\biggl[
e^{\frac{z}{2}}H_+\Big(\frac{5}{6},\frac{5}{3};z\Big)
+e^{-\frac{z}{2}}H_-\Big(\frac{5}{6},\frac{5}{3};z\Big)\biggr],
\end{split}
\label{equ:3.4}
\end{equation}
here $z=2i\phi(t)|\xi|$, $\xi\in\Bbb R$, $i=\sqrt{-1}$, and $H_{\pm}$ are smooth functions of the variable $z$.
By \cite{Tani}, one knows that for $\beta\in\mathbb{N}_0$,
\begin{align}
\big| \partial_\xi^\beta H_{+}(\alpha,\gamma;z)
\big|&\leq C(\phi(t)|\xi|)^{\alpha-\gamma}(1+|\xi|^2)^{-\frac{|\beta|}{2}}
\quad if \quad \phi(t)|\xi|\geq 1, \label{equ:3.5} \\
\big| \partial_\xi^\beta H_{-}(\alpha,\gamma;z)\big|&\leq C(\phi(t)|\xi|)^{-\alpha}(1+|\xi|^2)^{-\frac{|\beta|}{2}}
\quad if \quad \phi(t)|\xi|\geq 1. \label{equ:3.6}
\end{align}

To estimate $v$, it  only suffices to deal with $V_1(t, D_x)f(x)$  since the treatment
on $V_2(t, D_x)g(x)$ is similar. Indeed, if one just notices a simple fact of $t\phi(t)^{-\frac{5}{6}}=
C_1\phi(t)^{-\frac{1}{6}}$, it then follows from the expressions
of $V_1(t,\xi)$ and $V_2(t,\xi)$ that the orders of $t$ in $V_1(t,\xi)$ and $V_2(t,\xi)$ are the same.
Choose a cut-off function $\chi(s)\in C^{\infty}(\Bbb R)$
with $\chi(s)=
\left\{ \enspace
\begin{aligned}
1, \quad &s\geq2 \\
0, \quad &s\leq1
\end{aligned}
\right.$. Then
\begin{equation}
\begin{split}
V_1(t,|\xi|)\hat{f}(\xi)&=\chi(\phi(t)|\xi|)V_1(t,|\xi|)\hat{f}(\xi)+(1-\chi(\phi(t)|\xi|))V_1(t,|\xi|)\hat{f}(\xi) \\
&=:\hat{v}_1(t,\xi)+\hat{v}_2(t,\xi).
\end{split}
\label{equ:3.7}
\end{equation}
Together with \eqref{equ:3.3}, \eqref{equ:3.5} and \eqref{equ:3.6}, we derive that
\begin{equation}
{v}_1(t,x)=C\biggl(\int_{\mathbb{R}^n}e^{i(x\cdot\xi+\phi(t)|\xi|)}a_{11}(t,\xi)\hat{f}(\xi)\md\xi+
\int_{\mathbb{R}}e^{i(x\cdot\xi-\phi(t)|\xi|)}a_{12}(t,\xi)\hat{f}(\xi)\md\xi\biggr), \label{equ:3.8}
\end{equation}
where $C>0$ is a generic constant, and for $\beta\in\mathbb{N}_0$,
\begin{equation*}
\big| \partial_\xi^\beta a_{1l}(t,\xi)\big|\leq C_{l\beta}|\xi|^{-|\beta|}\big(1+\phi(t)|\xi|\big)^{-\frac{1}{6}},
\qquad l=1,2.
\end{equation*}

Next we analyze $v_2(t,x)$. It follows from \cite{Erd1} or \cite{Yag2} that
\begin{equation*}
V_1(t,|\xi|)=e^{-\frac{z}{2}}\Phi\Big(\frac{1}{6},\frac{1}{3};z\Big), %\label{equ:2.9}
\end{equation*}
where $\Phi$ is the confluent hypergeometric function which is analytic with respect to the variable
$z=2i\phi(t)|\xi|$. Then
\begin{equation*}
\Big|\partial_\xi\big\{\big(1-\chi(\phi(t)|\xi|)\big)V_1(t,|\xi|)\big\}\Big|\leq C(1+\phi(t)|\xi|)^{-\frac{1}{6}}|\xi|^{-1}.
\end{equation*}
Similarly, one has
\begin{equation*}
\Big|\partial_\xi^{\beta}\big\{\big(1-\chi(\phi(t)|\xi|)\big)V_1(t,|\xi|)\big\}\Big|\leq
C(1+\phi(t)|\xi|)^{-\frac{1}{6}}|\xi|^{-|\beta|}.
\end{equation*}
Thus we arrive at
\begin{equation}
v_2(t,x)=C\biggl(\int_{\mathbb{R}}e^{i(x\cdot\xi+\phi(t)|\xi|)}a_{21}(t,\xi)\hat{f}(\xi)\md\xi
+\int_{\mathbb{R}}e^{i(x\cdot\xi-\phi(t)|\xi|)}a_{22}(t,\xi)\hat{f}(\xi)\md\xi\biggr), \label{equ:3.9}
\end{equation}
where, for $\beta\in\mathbb{N}_0$,
\begin{equation*}
\big| \partial_\xi^\beta a_{2l}(t,\xi)\big|\leq C_{l\beta}\big(1+\phi(t)|\xi|\big)^{-\frac{1}{6}}|\xi|^{-|\beta|},
\qquad l=1,2.
\end{equation*}

Substituting \eqref{equ:3.8} and \eqref{equ:3.9} into \eqref{equ:3.7} yields

\[V_1(t, D_x)f(x)=C_1\biggl(\int_{\mathbb{R}}e^{i(x\xi+\phi(t)|\xi|)}a_1(t,\xi)\hat{f}(\xi)\md\xi
+\int_{\mathbb{R}}e^{i(x\xi-\phi(t)|\xi|)}a_2(t,\xi)\hat{f}(\xi)\md\xi\biggr),\]
where $a_l$ $(l=1,2)$ satisfies
\begin{equation}
|\partial_\xi^\beta a_l(t,\xi)\big|\leq C_{l\beta}\big(1+\phi(t)|\xi|\big)^{-\frac{1}{6}}|\xi|^{-|\beta|}. \label{equ:3.10}
\end{equation}

To estimate $V_1(t, D_x)f(x)$, it only suffices to deal with $\int_{\mathbb{R}}e^{i(x\xi+\phi(t)|\xi|)}a_1(t,\xi)\hat{f}(\xi)\md\xi$
since the term $\int_{\mathbb{R}}e^{i(x\xi-\phi(t)|\xi|)}a_2(t,\xi)\hat{f}(\xi)\md\xi$ can be  analogously treated.
Set
\begin{equation*}
(Af)(t,x)=:\int_{\mathbb{R}}e^{i(x\xi+\phi(t)|\xi|)}a_1(t,\xi)\hat{f}(\xi)\md\xi. %\label{equ:2.12}
\end{equation*}

Let $\beta(\tau)\in C_0^\infty(\frac{1}{2},2)$ such that
\begin{equation}
\sum\limits_{j=-\infty}^\infty\beta(\frac{\tau}{2^j}) \equiv1\quad \text{for $\tau\in\mathbb{R}_+$.} \label{equ:3.11}
\end{equation}

To estimate $(Af)(t,x)$, we now study its corresponding dyadic operators
\begin{equation*}
\begin{split}
(A_jf)(t,x)&=\int_{\mathbb{R}}e^{i(x\xi+\phi(t)|\xi|)}\beta(\frac{|\xi|}{2^j})
a_1(t,\xi)\hat{f}(\xi)\md\xi \\
&=:\int_{\mathbb{R}}e^{i(x\xi+\phi(t)|\xi|)}a_j(t,\xi)\hat{f}(\xi)\md\xi, \\
\end{split}
%\label{equ:2.14}
\end{equation*}
where $j\in\Bbb Z$. Note that the kernel of operator $A_j$ is
\begin{equation*}
K_j(t,x;y)=\int_{\mathbb{R}}e^{i((x-y)\xi+\phi(t)|\xi|)}a_j(t,\xi)\md\xi,
\end{equation*}
where $|y|\le M$ because of $\operatorname{supp} f\subseteq \{x: |x|\leq M\}$. By (3.29) of \cite{Ls}, we have that for any $N\in\Bbb {\Bbb R^+}$,
\begin{equation}\label{equ:3.12}
\begin{split}
|K_j(t,x;y)|\leq &C\lambda_j(1+\phi(t)\lambda_j)^{-\frac{1}{6}}\big(1+\lambda_j\big||x-y|-\phi(t)\big|\big)^{-N},
\end{split}
\end{equation}
where $\lambda_j=2^j$. Since the solution $v$ of \eqref{equ:3.1} is smooth and has compact support on the variable
$x$ for any fixed time, one easily knows that \eqref{equ:3.2} holds in any fixed domain $[0, T]\times\Bbb R$.
Therefore, in order to prove \eqref{equ:3.2}, it suffices to consider the case of $\phi(t)\gg M$.
At this time,
the following two cases will be studied separately.

\subsection{\boldmath$\big||x-y|-\phi(t)\big|\gg M$}\label{sec2:big}

For this case, there exist two positive constants $C_1$ and $C_2$ such that
\[C_1\big||x-y|-\phi(t)\big|\geq\big||x|-\phi(t)\big|\geq C_2\big||x-y|-\phi(t)\big|\gg M.\]
If $j\geq0$, we then take $N=\frac{5}{6}+\delta$ in \eqref{equ:3.12} and obtain
\begin{equation*}
\begin{split}
|K_j(t,x;y)|&\leq C_{\delta}\lambda_j^{1-\frac{1}{6}}
\phi(t)^{-\frac{1}{6}}\lambda_j^{-\frac{5}{6}-\delta}
\big||x|-\phi(t)\big|^{-\frac{5}{6}-\delta} \\
&\leq C_{\delta}\lambda_j^{-\delta}\big(1+\phi(t)\big)^{
-\frac{1}{6}}\big(1+\big||x|-\phi(t)\big|\big)^{-\frac{5}{6}-\delta}.
\end{split}
\end{equation*}
For $j<0$, taking $N=\frac{5}{6}-\delta$ in \eqref{equ:3.12}  we arrive at
\begin{equation*}
\begin{split}
|K_j(t,x;y)|&\leq C_{\delta}\lambda_j^{1-\frac{1}{6}}
\phi(t)^{-\frac{1}{6}}\lambda_j^{-\frac{5}{6}+\delta}
\big||x|-\phi(t)\big|^{-\frac{5}{6}+\delta} \\
&\leq C_{\delta}\lambda_j^\delta\big(1+\phi(t)\big)^{-\frac{1}{6}}\big(1+\big||x|-\phi(t)\big|\big)^{-\frac{1}{6}+\delta}.
\end{split}
\end{equation*}
It follows from $f(x)\in C_0^\infty(\mathbb{R})$ and direct computation that
\begin{equation} \label{equ:3.13}
|A_jf|\leq
\left\{ \enspace
\begin{aligned}
&C_{\delta}\lambda_j^\delta\big(1+\phi(t)\big)^{-\frac{1}{6}}
\big(1+\big||x|-\phi(t)\big|\big)^{-\frac{5}{6}+\delta}\|f\|_{L^1(\Bbb R)}, &&j<0,\\
&C_{\delta}\lambda_j^{-\delta}\big(1+\phi(t)\big)^{-\frac{1}{6}}
\big(1+\big||x|-\phi(t)\big|\big)^{-\frac{5}{6}-\delta}\|f\|_{L^1(\Bbb R)}, &&j\geq0.
\end{aligned}
\right.
\end{equation}

Summing the right sides of \eqref{equ:3.13}, we get that for large $\phi(t)$ and $\big||x|-\phi(t)\big|$,
\begin{equation}\label{equ:3.14}
|V_1(t,D_x)f|\leq C_{\delta}\big(1+\phi(t)\big)^{-\frac{1}{6}}
\big(1+\big||x|-\phi(t)\big|\big)^{-\frac{5}{6}+\delta}\|f\|_{L^1(\Bbb R)}.
\end{equation}
Analogously, we have
\begin{equation*}\label{equ:3.14.1}
|V_2(t,D_x)g|\leq C_{\delta}\big(1+\phi(t)\big)^{-\frac{1}{6}}
\big(1+\big||x|-\phi(t)\big|\big)^{-\frac{1}{6}+\delta}\|g\|_{L^1(\Bbb R)}.
\end{equation*}

Therefore,
\begin{equation}\label{Y-0}
|v|\leq C_{\delta}\big(1+\phi(t)\big)^{-\frac{1}{6}}
\big(1+\big||x|-\phi(t)\big|\big)^{-\frac{1}{6}+\delta}(\|f\|_{L^1(\Bbb R)}+\|g\|_{L^1(\Bbb R)}).
\end{equation}

\subsection{\boldmath$\big||x-y|-\phi(t)\big|\le C M$}\label{sec2:small}

By the similar method as in {\bf 2.1}, we can establish that for $t>1$,
\begin{equation}\label{equ:3.15}
\parallel v(t,\cdot)\parallel_{L^\infty(\mathbb{R})}\leq C_{\delta}\phi(t)^{-\frac{1}{6}}\left(\parallel f\parallel_{W^{\frac{5}{6}+\delta,1}(\mathbb{R})}
+\parallel g\parallel_{W^{\frac{1}{6}+\delta,1}(\mathbb{R})}\right),
\end{equation}
where $0<\delta<\frac{5}{6}-\gamma-\frac{1}{q}$ is a constant.

Indeed, note that
\begin{equation*}
\begin{split}
|A_jf|&=\bigg|\int_{\mathbb{R}}e^{i(x\cdot\xi+\phi(t)|\xi|)}\frac{a_j(t,\xi)}{|\xi|^\alpha}\widehat{|D_x|^\alpha f}(\xi)\md\xi\bigg|,
\end{split}
\end{equation*}
where $\alpha=\frac{5}{6}+\delta$. Then by direct computation, we have
that for $j\ge 0$,
\begin{equation}\label{equ:3.16}
\begin{split}
|A_jf|&\leq C_{\delta}\lambda_j^{-\alpha}\lambda_j
\big(1+\phi(t)\lambda_j\big)^{-\frac{1}{6}}
\parallel f\parallel_{W^{\frac{5}{6}+\delta,1}(\mathbb{R})} \\
&\leq C_{\delta}\lambda_j^{-\delta}\big(1+\phi(t)\big)^{-\frac{1}{6}}
\parallel f\parallel_{W^{\frac{5}{6}+\delta,1}(\mathbb{R})}.
\end{split}
\end{equation}
Similarly, for $j<0$, we have
\begin{equation}\label{equ:3.17}
|A_jf|\leq C_{\delta}\lambda_j^\delta\big(1+\phi(t)\big)^{-\frac{1}{6}}
\|f\|_{W^{\frac{5}{6}-\delta,1}(\mathbb{R})}.
\end{equation}

Summing all the terms in \eqref{equ:3.16} and \eqref{equ:3.17} yields \eqref{equ:3.15} for $g=0$.
Note that $\big||x-y|-\phi(t)\big|\le C M$ and \(\operatorname{supp}f\subseteq[-M,M]\), we then have
\[\big||x|-\phi(t)\big|\le C M.\]
This together with \eqref{equ:3.15}  for $g=0$ yields
\begin{equation}\label{equ:3.20}
  |V_1(t,D_x)f|\leq C_{\delta}\big(1+\phi(t)\big)^{-\frac{1}{6}}
\big(1+\big||x|-\phi(t)\big|\big)^{-\frac{5}{6}+\delta}\|f\|_{W^{\frac{5}{6}+\delta,1}(\mathbb{R})}.
\end{equation}

Analogously, we have
\begin{equation}\label{equ:3.20.1}
  |V_2(t,D_x)g|\leq C_{\delta}\big(1+\phi(t)\big)^{-\frac{1}{6}}
\big(1+\big||x|-\phi(t)\big|\big)^{-\frac{1}{6}+\delta}\|g\|_{W^{\frac{1}{6}+\delta,1}(\mathbb{R})}.
\end{equation}
Therefore, it follows from \eqref{equ:3.20} and \eqref{equ:3.20.1} that
\begin{equation}\label{equ:3.18}
|v|\leq C_{\delta}(1+\phi(t))^{-\frac{1}{6}}(1+\big||x|-\phi(t)\big|)^{-\frac{1}{6}+\delta}
\left(\parallel f\parallel_{W^{\frac{5}{6}+\delta,1}(\mathbb{R})}
+\parallel g\parallel_{W^{\frac{1}{6}+\delta,1}(\mathbb{R})}\right).
\end{equation}

Combining \eqref{equ:3.15} with \eqref{equ:3.18} and noting the compact supports of $f,g$, we have
\begin{equation}\label{Y-1}
|v|\leq C_{\delta}(1+\phi(t))^{-\frac{1}{6}}(1+\big||x|-\phi(t)\big|)^{-\frac{1}{6}+\delta}
\left(\parallel f\parallel_{W^{\frac{5}{6}+\delta,1}(\mathbb{R})}
+\parallel g\parallel_{W^{\frac{1}{6}+\delta,1}(\mathbb{R})}\right).
\end{equation}

Next we derive \eqref{equ:3.2} from \eqref{Y-1}. Set
\[R=\parallel f\parallel_{W^{\frac{5}{6}+\delta,1}(\mathbb{R})}
+\parallel g\parallel_{W^{\frac{1}{6}+\delta,1}(\mathbb{R})}.\]
Then
\begin{equation}\label{equ:3.19}
\begin{split}
&\Big\|\big((\phi(t)+M)^2-|x|^2\big)^\gamma v\Big\|_{L^q(\mathbb{R}^{1+1}_+)}^q \\
&\le C_{\delta}R\int_0^\infty\int_{\mathbb{R}}
\bigg(\Big(\big(\phi(t)+M\big)^2-|x|^2\Big)^\gamma\big(1+\phi(t)\big)^{-\frac{1}{6}}\Big(1+\big||x|-\phi(t)\big|\Big)^{-\frac{1}{6}+\delta}\bigg)^q \md x\md t
\\
&\leq C_{\delta}R
\int_0^\infty\int_0^\infty\Big(\big(\phi(t)+M+r\big)^\gamma\big(\phi(t)+M-r\big)^\gamma \\
&\qquad\qquad \qquad  \qquad  \qquad  \times\big(1+\phi(t)\big)^{-\frac{1}{6}}
\big(1+|r-\phi(t)|\big)^{-\frac{1}{6}+\delta}\Big)^q \md r\md t
\\
&\leq C_{m,\delta}R
\int_0^\infty\int_0^\infty\Big(\big(1+\phi(t)\big)^{-\frac{1}{6}+\gamma}\big(1+|r-\phi(t)|\big)^{\gamma-\frac{1}{6}+\delta}\Big)^q
\md r\md t \\
&\leq C_{m,\delta}R
\int_0^\infty\big(1+\phi(t)\big)^{\big(-\frac{1}{3}+2\gamma+\delta\big)q+1}\md t.
\end{split}
\end{equation}
Notice that by our assumption, $\gamma-\frac{1}{6}<-\frac{5}{6q}$ holds.
Then we can choose a constant $\dl>0$ such that
\[\Big(\big(-\frac{1}{3}+2\gamma+\delta\big)q+1\Big)\frac{3}{2}<-1.\]
Hence, for some positive constant $\sigma>0$, the integral in the last line of \eqref{equ:3.19} can be controlled by
\begin{equation*}
\begin{split}
&\int_0^\infty\big(1+\phi(t)\big)^{\big(-\frac{1}{3}+2\gamma+\delta\big)q+1}\md t \\
&\leq C\int_0^\infty(1+t)^{-1-\sigma}\md t\\
&\leq C.
\end{split}
\end{equation*}
This, together with \eqref{equ:3.19}, yields \eqref{equ:3.2}. Namely, we complete the proof of Theorem 2.1.
\end{proof}

\section{Mixed-norm estimate for inhomogeneous equation}

In this section we turn to the inhomogeneous Tricomi equation:
\begin{equation}
\begin{cases}
&\partial_t^2 w-t\partial_x^2 w=F(t,x), \quad \text{in} \quad \R_+^{1+1},\\
&w(0,x)=0,\quad \partial_tw(0,x)=0.
\end{cases}
\label{equ:4.1}
\end{equation}

Since the stationary phase method is not applicable
for the solution $w$ in the case of
\(n=1\), we can not get a suitable \(L^1-L^\infty\) estimate and then use interpolation to get the Strichartz-type estimate
as in \cite{HWYin2}. To overcome this difficulty, we shall cite a conclusion from \cite{Gls} and subsequently
use the representation formula of $w$ to establish the space-time mixed norm estimate by delicate analysis.

\begin{lemma}{\bf(see (1.16) of \cite{Gls})}\label{lem:3.1}
If
\[f(u)=\int_0^u\frac{g(\xi)}{|u-\xi|^\delta|\xi|^\beta|u|^{\alpha}}d\xi,\]
then
\begin{equation}
\|f\|_{L^q((0,\infty))}\leq C\|g\|_{L^r((0,\infty))},
\end{equation}
where
\[1<r<q<\infty, \alpha+\beta+\delta=1-\left(\frac{1}{r}-\frac{1}{q}\right), \alpha+\beta\geq0, \quad and \quad \alpha+\delta>\frac{1}{q}.\]
\end{lemma}

\begin{theorem}\label{thm:3.2}
For problem \eqref{equ:4.1}, if $F(t,x)\equiv0$ when
$|x|>\phi(t)-1$, then there exist some constants $\alpha$ and $\beta$ satisfying
\begin{equation}\label{Y-2}
\alpha+\frac{1}{6}+\beta=\f{5}{3q}, \quad \beta<\frac{1}{q},
\end{equation}
such that
\begin{equation}
\begin{split}
\big\|\big(\phi(t)^2-|x|^2\big)^{-\alpha}w\big\|_{L^q(\mathbb{R}^{1+1}_+)}\leq C
\big\|\big(\phi(t)^2-|x|^2\big)^{\beta}F\big\|_{L^{\frac{q}{q-1}}(\mathbb{R}^{1+1}_+)},
\end{split}
\label{equ:4.2}
\end{equation}
where $q=1+p$, \(p_{\crit}<p<p_0=9\) and $C>0$ is a constant depending on $m$, $q$, $\alpha$ and $\beta$.
\end{theorem}

\begin{remark}
  Recall that in \cite{HWYin1} we have defined the conformal exponent $p_{\conf}(n)$ for $n$-dimensional semilinear Tricomi equation
  $\p_t^2u-t\Delta u=|u|^p$
  \[p_{\conf}(n)=\f{N+2}{N-2}=\frac{3n+6}{3n-2}.\]
 Set $n=1$, we then have \(p_{\conf}(1)=p_0=9\).
\end{remark}

\begin{proof}
By the formula in Theorem 2.4 of \cite{Yag3}, the solution $w$ of \eqref{equ:4.1} satisfies
\begin{equation}\label{equ:4.3}
\begin{split}
w(t,x)=C\int_0^t\int_{x-\phi(t)+\phi(s)}^{x+\phi(t)-\phi(s)}&\big(\phi(t)+\phi(s)+x-y\big)^{-\gamma}
\big(\phi(t)+\phi(s)-(x-y)\big)^{-\gamma}\\
&\times H\big(\gamma,\gamma,1,z\big)F(s,y)\ \md y\md s,
\end{split}
\end{equation}
where $z=\frac{(-x+y+\phi(t)-\phi(s))(-x+y-\phi(t)+\phi(s))}
{(-x+y+\phi(t)+\phi(s))(-x+y-\phi(t)-\phi(s))}$, $H\big(\gamma,\gamma,1,z\big)$
is the hypergeometric function and \(\gamma=\f{1}{6}\).
%By the domain of dependance, we know that \(|x-y|\leq\phi(t)-\phi(s)\) in the support of \(w\) and \(F\). Thus
%\begin{multline*}
%\big(\phi(t)+\phi(s)+x-y\big)^{-\gamma}\big(\phi(t)+\phi(s)-(x-y)\big)^{-\gamma} \\
%\begin{split}
%&=\Big(\big(\phi(t)+\phi(s)\big)^2-(x-y)^2\Big)^{-\gamma} \\
%& \leq\Big(\big(\phi(t)+\phi(s)\big)^2-(\phi(t)-\phi(s))^2\Big)^{-\gamma}\\
%& =\big(2\phi(t)\phi(s)\big)^{-\gamma}
%\end{split}
%\end{multline*}
By Page 59 of \cite{Erd1},
\begin{equation*}
\begin{split}
H(\gamma,\gamma,1,z)=&\frac{1}{\Gamma(\gamma)\Gamma(1-\gamma)}\int_0^1t^{\gamma-1}(1-t)^{-\gamma}
(1-zt)^{-\gamma}\md t \\
\leq &\frac{1}{\Gamma(\gamma)\Gamma(1-\gamma}\int_0^1t^{\gamma-1}(1-t)^{-\gamma}(1-t)^{-\gamma}\md t \\
= &\frac{1}{\Gamma(\gamma)\Gamma(1-\gamma)}B(\gamma, 1-2\gamma)\\
= &C.
\end{split}
\end{equation*}
Thus we have
\begin{equation}\label{equ:4.4}
  |w(t,x)|\leq C\int_0^t\int_{x-\phi(t)+\phi(s)}^{x+\phi(t)-\phi(s)}\big(\phi(t)+\phi(s)+x-y\big)^{-\gamma}
\big(\phi(t)+\phi(s)-(x-y)\big)^{-\gamma}|F(s,y)|\ \md y\md s.
\end{equation}
By \eqref{equ:4.2}, we need to estimate
\begin{equation}\label{equ:4.5}
  \begin{split}
\Big\|\big(\phi^2(t)-|x|^2\big)^{-\alpha}&w\Big\|_{L^q(\mathbb{R}^{1+1}_+)} \\
&=\sup_{K\in L^{\frac{q}{q-1}}(\mathbb{R}^{1+1}_+)}
\frac{\big|\int_0^{\infty}\int_{-\infty}^{\infty} K(t,x)
(\phi^2(t)-|x|^2)^{-\alpha}w\md x\md t\big|}{\parallel K\parallel_{L^{\frac{q}{q-1}}(\mathbb{R}^{1+1}_+)}}.
\end{split}
\end{equation}
Note that
\begin{equation}\label{equ:4.6}
\begin{split}
I& =:\int_0^{\infty}\int_{-\infty}^{\infty} K(t,x)
(\phi^2(t)-|x|^2)^{-\alpha}w(t,x)\ \md x\md t \\
& =\int_0^{\infty}\int_{-\infty}^{\infty}\int_0^t\int_{x-(\phi(t)-\phi(s))}^{x+\phi(t)-\phi(s)}
\frac{K(t,x)F(s,y)}{(\phi^2(t)-|x|^2)^\alpha} \\
&\quad \times\frac{1}{\big(\phi(t)+\phi(s)+x-y\big)^{\gamma}
\big(\phi(t)+\phi(s)-(x-y)\big)^{\gamma}}\ \md y \md s \md x \md t.
\end{split}
\end{equation}
Denote $\tilde{K}(T,x)=K(t,x)$ and $\tilde{F}(S,y)=F(s,y)$ with $T=\phi(t)$ and \(S=\phi(s)\). Then
\begin{equation}\label{equ:4.9}
\begin{split}
\parallel K\parallel_{L^{\frac{q}{q-1}}(\mathbb{R}^{1+1}_+)}
& =\bigg(\int_0^{\infty}\Big(\int_{-\infty}^{\infty}|K(t,x)|^{\frac{q}{q-1}}
\ \md x\Big)\ \md t\bigg)^{\frac{q-1}{q}} \\
& =\bigg(\int_0^{\infty}\Big(\int_{-\infty}^{\infty}|T^{-\frac{1}{6}\cdot\frac{q-1}{q}}\tilde{K}(T,x)|^{\frac{q}{q-1}}
\ \md x\Big)\ \md T\bigg)^{\frac{q-1}{q}}
\end{split}
\end{equation}
and
\begin{equation}\label{equ:4.10}
\begin{split}
&\Big\|\big(\phi^2(t)-|x|^2\big)^{\beta} F\Big\|_{L^{\frac{q}{q-1}}(\mathbb{R}^{1+1}_+)} \\
=&\bigg(\int_0^{\infty}\Big(\int_{-\infty}^{\infty}|T^{-\frac{1}{6}\cdot\frac{q-1}{q}}
(T^2-|x|^2)^{\beta}\tilde{F}(T,x)|^{\frac{q}{q-1}}\ \md x\Big)\ \md T\bigg)^{\frac{q-1}{q}}.
\end{split}
\end{equation}
With \eqref{equ:4.9} and \eqref{equ:4.10}, we can further write
\begin{equation}\label{equ:4.7}
\begin{split}
I& =\int_0^{\infty}\int_{-\infty}^{\infty}\int_0^T\int_{x-(T-S)}^{x+T-S}
\frac{\tilde{K}(T,x)\tilde{F}(S,y)}{(T^2-x^2)^{\alpha}} \\
&\qquad \times\frac{1}{\big(T+S+x-y\big)^{\gamma}
\big(T+S-(x-y)\big)^{\gamma}}S^{-\frac{1}{6}}T^{-\frac{1}{6}}\ \md y \md S \md x \md T \\
& =\int_0^{\infty}\int_{-\infty}^{\infty}\int_0^T\int_{x-(T-S)}^{x+T-S}
\frac{T^{-\frac{1}{6}\cdot\frac{1}{p}}|\tilde{K}(T,x)|S^{-\frac{1}{6}\cdot\frac{1}{p}}
(S^2-y^2)^{\beta}|\tilde{F}(S,y)|}
{(S^2-y^2)^{\beta}(T^2-x^2)^{\alpha}S^{\frac{1}{3q}}
T^{\frac{1}{3q}}} \\
&\qquad \times\frac{1}{\big(T+S+x-y\big)^{\gamma}
\big(T+S-(x-y)\big)^{\gamma}}\md y \md S \md x \md T.
\end{split}
\end{equation}

Let
\begin{equation*}
  \left\{ \enspace
\begin{aligned}
&u=T+x,\\
&v=T-x,
\end{aligned}
\right.
\qquad\qquad
  \left\{ \enspace
\begin{aligned}
&\xi=S+y,\\
&\eta=S-y.
\end{aligned}
\right.
\end{equation*}
By the assumption of $\text{supp$F(t,x)$}$, we know that
\[0\leq\xi\leq u, 0\leq\eta\leq v.\]

Set
\begin{equation*}\label{equ:4.11}
  G(\xi,\eta)=S^{-\frac{1}{6}\cdot\frac{1}{p}}(S^2-y^2)^{\beta}|\tilde{F}(S,y)|,
\end{equation*}
\begin{equation*}
  H(u,v)=T^{-\frac{1}{6}\cdot\frac{1}{p}}|\tilde{K}(T,x)|.\label{equ:4.12}
\end{equation*}
By \(T\geq S\geq \phi(1)\), we then have
\begin{equation}\label{equ:4.8}
  \begin{split}
     I & \leq\iint_{0\leq\xi\leq u}\iint_{0\leq\eta\leq v}\f{G(\xi,\eta)H(u,v)}
     {\xi^\beta\eta^\beta u^\alpha v^\alpha (u+\xi)^{\frac{1}{6}}(v+\eta)^{\frac{1}{6}}S^{\frac{1}{3q}}T^{\frac{1}{3q}}}\md y\md S\md x\md T \\
       & \leq\iint_{0\leq\xi\leq u}\iint_{0\leq\eta\leq v}\f{G(\xi,\eta)H(u,v)}
     {\xi^\beta u^\alpha|u-\xi|^{\frac{1}{3q}+\frac{1}{6}}\eta^\beta v^\alpha|v-\eta|^{\frac{1}{3q}+\frac{1}{6}}}\md \eta\md v\md \xi\md u.
  \end{split}
\end{equation}

By conditions
\eqref{Y-2} in Theorem~\ref{thm:3.2}, we require
\[1<p<2<q<\infty \]
and
\begin{equation}\label{equ:4.8.1}
\alpha+\beta=\f{2}{q}-\frac{1}{3q}-\frac{1}{6}\geq0, \quad \beta<\frac{1}{q}.
\end{equation}
Choosing \(\beta=-p\alpha\) in \eqref{equ:4.8.1}, then by
\(\beta<\frac{1}{q}\) and \(q=p+1\) one has that
\begin{equation}\label{equ:4.8.2}
  -p\alpha<\frac{1}{p+1}\Rightarrow\alpha>-\frac{1}{p(p+1)}.
\end{equation}
Substitute \eqref{equ:4.8.2} into \eqref{equ:4.8.1}, we get
\[\frac{p-1}{p(p+1)}+\frac{1}{6}\left(\frac{1}{2}+\frac{1}{q}\right)>\frac{2}{p+1}
\Leftrightarrow p^2-3p-6>0\Rightarrow p>p_1=\frac{3+\sqrt{33}}{2}.\]
%Note that this condition corresponds to the condition \eqref{equ:2.18} for the blowup part.

On the other hand,
\[\frac{2}{q}-\frac{1}{3q}-\frac{1}{6}\geq0\Leftrightarrow q\leq10
\Leftrightarrow p\leq p_0=9.\]
Thus Lemma~\ref{lem:3.1} is applicable for $\dl=\frac{1}{3q}+\frac{1}{6}$ and \(p=\f{q}{q-1}\), we then estimate the
integral in \eqref{equ:4.8} as follows
\begin{equation*}
\begin{split}
\iint_{0\leq\xi\leq u}\iint_{0\leq\eta\leq v}&\f{G(\xi,\eta)H(u,v)}
     {|\xi|^\beta|u|^\alpha
|u-\xi|^{\frac{1}{3q}+\frac{1}{6}}|\eta|^\beta|v|^\alpha
|v-\eta|^{\frac{1}{3q}+\frac{1}{6}}}\md \eta\md v\md \xi\md u \\
&=\int_0^\infty\int_0^u\frac{1}{|\xi|^\beta|u|^\alpha
|u-\xi|^{\frac{1}{3q}+\frac{1}{6}}}\int_0^\xi H(u,v)\int_0^v
\frac{G(\xi,\eta)}{|\eta|^\beta|v|^\alpha
|v-\eta|^{\frac{1}{3q}+\frac{1}{6}}}
\md \eta\md v\md \xi\md u \\
&\leq\int_0^\infty\int_0^u\frac{1}{|\xi|^\beta|u|^\alpha
|u-\xi|^{\frac{1}{3q}+\frac{1}{6}}} \|H(u,\cdot)\|_{L^p}
\bigg\|\int_0^v
\frac{G(\xi,\eta)}{|\eta|^\beta|v|^\alpha
|v-\eta|^{\frac{1}{3q}+\frac{1}{6}}}
d\eta\bigg\|_{L_v^q}d\xi d\eta \\
& \leq\int_0^\infty\|H(u,\cdot)\|_{L^p}\int_0^u
\frac{\|G(\xi,\cdot)\|_{L^p}}{|\xi|^\beta|u|^\alpha
|u-\xi|^{\frac{1}{3q}+\frac{1}{6}}}
d\xi du \\
&\leq\|H\|_{L^p_{u,v}}\bigg\|\int_0^u
\frac{\|G(\xi,\cdot)\|_{L^p}}{|\xi|^\beta|u|^\alpha
|u-\xi|^{\frac{1}{3q}+\frac{1}{6}}}
d\xi\bigg\|_{L_u^q} \\
&\leq\|H\|_{L^p_{u,v}}\|G\|_{L^p_{\xi,\eta}}.
\end{split}
\end{equation*}
This, together with \eqref{equ:4.5} and \eqref{equ:4.8}, yields the proof of Theorem~\ref{thm:3.2}.

\end{proof}

Based on Theorem~\ref{thm:3.2}, we are able to prove such a crucial result:
\begin{theorem}\label{thm:3.3}
For problem \eqref{equ:4.1}, if $F(t,x)\equiv0$ when
$|x|>\phi(t)+M-1$ and $F\in C^{\infty}([0, T_0]\times\Bbb R)$ for some fixed
number $T_0>0$ $(T_0<1)$, then there exist some constants $\alpha$ and \(\beta\)satisfying $\alpha+\frac{1}{6}+\beta=\f{5}{3q}$, $\beta>\frac{1}{q}$, such that
\begin{equation}
\begin{split}
\Big\|\Big(\big(\phi(t)+M\big)^2-|x|^2\Big)^{-\alpha}w\Big\|_{L^q([\f{T_0}{2}, \infty)\times \mathbb{R})}\leq C
\Big\|\Big(\big(\phi(t)+M\big)^2-|x|^2\Big)^\beta F\Big\|_{L^{\frac{q}{q-1}}([\f{T_0}{2}, \infty)\times \mathbb{R})},
\end{split}
\label{equ:4.13}
\end{equation}
where $q=1+p$, \(p_{\crit}<p<9\), and $C>0$ is a constant depending on $q$, $\alpha$ and $\beta$.
\end{theorem}

\begin{proof}
To prove \eqref{equ:4.13}, at first we focus on a special case of $F(t,x)\equiv0$ when $|x|>\phi(t)-\phi\big(\frac{T_0}{4}\big)$.
By the finite propagation speed property for the hyperbolic equation \eqref{equ:4.1}, we know that
the integral domain in  \eqref{equ:4.13} is just only  $Q=:\{(t,x): t\geq\frac{T_0}{2}, |x|\leq\phi(t)+M-1\}$.
Note that $Q$ can be covered by a finite number of angular domains $\{Q_j\}_{j=1}^{N_0}$, where the curved cone $Q_j$
$(j\ge 2)$
is a shift in the $x$ variable with respect to the angular domain
\[Q_1=\Big\{(t,x): t\ge\frac{T_0}{2}, |x|\leq\phi(t)-\phi\Big(\frac{T_0}{4}\Big)\Big\}.\]
Set
\begin{equation*}
\begin{split}
&F_1=\chi_{Q_1}F, \\
&F_2=\chi_{Q_2}(1-\chi_{Q_1})F,\\
&\qquad \qquad \ldots\\
&F_{N_0}=\chi_{Q_{N_0}}\big(1-\chi_{Q_1}-\chi_{Q_2}(1-\chi_{Q_1})-\cdot\cdot\cdot
-\chi_{Q_{N_0-1}}(1-\chi_{Q_1})\cdot\cdot\cdot(1-\chi_{Q_{N_0-2}})\big)F, \\
\end{split}
\end{equation*}
where $\chi_{Q_j}$ stands for the characteristic function of $Q_j$, and $\ds\sum_{j=1}^{N_0}F_j=F$. Let $w_j$ solve
\begin{equation*}
\begin{cases}
&\partial_t^2 w_j-t\triangle w_j=F_j(t,x), \\
&w_j(0,x)=0,\quad \partial_tw_j(0,x)=0.
\end{cases}
\end{equation*}
Then $\operatorname{supp}w_j\subseteq Q_j$. Since the Tricomi equation is invariant under the translation
with respect to the variable $x$, it follows from Theorem 3.2 that
\begin{equation}\label{equ:4.14}
\big\|\big(\phi(t)^2-|x-\nu_j|^2\big)^{\gamma_1}w_j\big\|_{L^q(Q_j)}\leq C
\big\|\big(\phi(t)^2-|x-\nu_j|^2\big)^{\gamma_2}F_j\big\|_{L^{\frac{q}{q-1}}(Q_j)},
\end{equation}
where $\nu_j\in\mathbb{R}^n$ corresponds to the coordinate shift of the space variable $x$ from $Q_1$ to $Q_j$,
and $Q_j=\{(t,x): t\ge\frac{T_0}{2},
|x-\nu_j|\leq\phi(t)-\phi(\frac{T_0}{4})\}$.

Next we derive \eqref{equ:4.13} by utilizing \eqref{equ:4.14} and the condition of $t\geq\frac{T_0}{4}$.
At first, we illustrate that there exists a constant $\delta>0$ such that for $(t,x)\in Q_j$,
\begin{equation}\label{equ:4.15}
\phi(t)^2-|x-\nu_j|^2\geq\delta\Big(\big(\phi(t)+M\big)^2-|x|^2\Big).
\end{equation}

To prove \eqref{equ:4.15} for $1\le j\le N_0$, it only suffices to consider the two extreme cases: $\nu_j=0$ (corresponding to
$j=1$) and $|\nu_{j_0}|=M-1+\phi(\frac{3T_0}{8})$ (choosing  $j_0$
such that $|\nu_{j_0}|=\max_{1\le j\le N_0}|\nu_j|=M-1+\phi(\frac{3T_0}{8})$. Note that
$|\nu_{j_0}|>M-1$ holds so that the domain $Q$ can be covered by $\ds{\cup_{j=1}^{N_0}}Q_j$).

For $\nu_j=0$, \eqref{equ:4.15} is equivalent to
\begin{equation}\label{equ:4.16}
\phi(t)^2\geq(1-\delta)|x|^2+\delta\big(\phi(t)+M\big)^2.
\end{equation}
We now illustrate that \eqref{equ:4.16} is correct.
By $|x|\leq\phi(t)-\phi(\frac{T_0}{4})$ for $(t,x)\in Q_1$,
then in order to show \eqref{equ:4.16} it suffices to prove
\begin{equation*}
\begin{split}
\phi(t)^2\geq(1-\delta)\bigg(\phi(t)-\phi\Big(\frac{T_0}{4}\Big)\bigg)^2+\delta\big(\phi(t)+M\big)^2.
\end{split}
\end{equation*}
This is equivalent to
$$\Big\{2(1-\delta)\phi\Big(\frac{T_0}{4}\Big)-2\delta M\Big\}\phi(t)\geq(1-\delta)\phi^2\Big(\frac{T_0}{4}\Big)+\delta M^2.
$$
Obviously, this is easily achieved by $t\geq \frac{T_0}{4}$ and the smallness of $\delta$.

For $\nu_{j_0}=M-1+\phi\big(\frac{3T_0}{8}\big)$, the argument on \eqref{equ:4.15} is a little involved. First, note that
for fixed $t>0$, the domain
$Q$ is symmetric with respect to the variable $x$,
thus we can assume $\nu_{j_0}=\nu=M-1+\phi(\frac{3T_0}{8})$.
In this case, \eqref{equ:4.15} is equivalent to
\begin{equation}\label{equ:4.17}
\begin{split}
\phi(t)^2&\geq|x-\nu|^2+\delta\big((\phi(t)+M)^2-|x|^2\big) \\
&=(1-\delta)x^2-2\nu x+\nu^2+\delta\big(\phi(t)+M\big)^2 \\
&=: G(t,x).
\end{split}
\end{equation}
For fixed $t>0$, $G(t,x)$ is a quadratic function of the variable $x$, and
takes minimum at the point $x=\f{\nu}{1-\delta}$. Thus for the same fixed $t>0$,
the maximum of $G(t,x)$ in the domain $Q^t=:\{x: |x-\nu|\leq\phi(t)-\phi(\frac{T_0}{4})\}$
must be achieved on the boundary $\p Q^t=:\{x: |x-\nu|=\phi(t)-\phi\big(\frac{T_0}{4}\big)\}$.
Then in order to show
\eqref{equ:4.17}, our task is to prove
\begin{equation}\label{equ:4.18}
\phi(t)^2\geq\bigg(\phi(t)-\phi\Big(\frac{T_0}{4}\Big)\bigg)^2+\delta\Big(\big(\phi(t)+M\big)^2-|x|^2\Big).
\end{equation}
For this end, it is only enough to consider the case that $|x|^2$ takes its minimum on $\p Q^t$.
Note that on $\p Q^t$, we have
\begin{equation}\label{equ:4.19}
|x|^2=\bigg(\phi(t)-\phi\Big(\frac{T_0}{4}\Big)\bigg)^2+2\nu x-\nu^2.
\end{equation}
Therefore, without loss of generality, we can take
\begin{equation}\label{equ:4.20}
x=\nu-\phi(t)+\phi\Big(\frac{T_0}{4}\Big).
\end{equation}
Substituting \eqref{equ:4.20} and \eqref{equ:4.19} into \eqref{equ:4.18}, we are left to prove
\begin{equation}\label{equ:4.21}
\begin{split}
\phi(t)^2&\geq\bigg(\phi(t)-\phi\Big(\frac{T_0}{4}\Big)\bigg)^2+\delta\bigg\{\big(\phi(t)+M\big)^2 \\
&\quad -\bigg(\phi(t)-\phi\Big(\frac{T_0}{4}\Big)\bigg)^2+2\nu\bigg(\phi(t)-\phi\Big(\frac{T_0}{4}\Big)\bigg)-\nu^2\bigg\} \\
&=\phi(t)^2+\bigg\{2\delta\bigg(\phi\Big(\frac{T_0}{4}\Big)+M+\nu\bigg)-2\phi\Big(\frac{T_0}{4}\Big)\bigg\}\phi(t) \\
&\quad +(1-\delta)\phi\Big(\frac{T_0}{4}\Big)^2+\delta M^2-\delta\nu\bigg(\nu+2\phi\Big(\frac{T_0}{4}\Big)\bigg).
\end{split}
\end{equation}
For fixed $T_0>0$ and $M>1$, if $\delta>0$ is small enough, one then has
\begin{equation}\label{equ:4.22}
\begin{split}
&2\delta\bigg(\phi\Big(\frac{T_0}{4}\Big)+M+\nu\bigg)\leq\f{1}{2}\phi\Big(\frac{T_0}{4}\Big), \\
&(1-\delta)\phi\Big(\frac{T_0}{4}\Big)^2+\delta M^2\leq\f{3}{2}\phi\Big(\frac{T_0}{4}\Big)^2.
\end{split}
\end{equation}
By \eqref{equ:4.22} and \eqref{equ:4.21}, in order to derive \eqref{equ:4.18}, one should
derive
\[-\f{3}{2}\phi\Big(\frac{T_0}{4}\Big)\phi(t)+\f{3}{2}\phi^2\Big(\frac{T_0}{4}\Big)\leq0.\]
Obviously, this holds true by $t\geq\frac{T_0}{4}$. Then \eqref{equ:4.18} is proved.

Consequently, for $(t,x)\in\bigcup_{j=1}^{N_0}Q_j$,
there exists a fixed positive constant $c>0$ such that for $1\le j\le N_0$,
\begin{equation}\label{equ:4.23}
\begin{split}
c\Big(\big(\phi(t)+M\big)^2-|x|^2\Big)\leq\phi(t)^2-|x-\nu_j|^2.
\end{split}
\end{equation}

On the other hand, note that by $|x|\leq\phi(t)+M-1$ for $(t,x)\in Q$, one has
\begin{equation}\label{equ:4.24}
\begin{split}
&2\big\{\big(\phi(t)+M\big)^2-|x|^2\big\}-\{\phi^2(t)-|x-\nu_j|^2\} \\
&\geq (|x|+1)^2-|x|^2+\big(\phi(t)+M\big)^2-|x|^2-\phi(t)^2+|x-\nu_j|^2 \\
&=2M\phi(t)+M^2+|\nu_j|^2+1+2(1-|\nu_j|)|x|.
\end{split}
\end{equation}
If $1-|\nu_j|<0$, then by $|\nu_j|\leq M-1+\phi\big(\frac{3T_0}{8}\big)$ and the smallness of
$T_0$, the last line in \eqref{equ:4.24} is bounded from below by
\begin{equation}\label{equ:4.25}
\begin{split}
&2M\phi(t)+M^2+|\nu_j|^2+1+2\Big\{2-M-\phi\Big(\frac{3T_0}{8}\Big)\Big\}\{\phi(t)+M-1\} \\
&=4\phi(t)-M^2+6M-3+|\nu_j|^2-2\phi\Big(\frac{3T_0}{8}\Big)\phi(t)-2(M-1)\phi\Big(\frac{3T_0}{8}\Big) \\
&\geq 2\phi(t)-M^2+1;
\end{split}
\end{equation}
while in the case of $1-|\nu_j|\geq0$, it follows from \eqref{equ:4.24}  that
\begin{equation}\label{equ:4.26}
2\big\{\big(\phi(t)+M\big)^2-|x|^2\big\}-\{\phi(t)^2-|x-\nu_j|^2\}\geq M^2+1>0.
\end{equation}
Substituting \eqref{equ:4.25}-\eqref{equ:4.26} into \eqref{equ:4.24} yields that for $2\phi(t)\geq M^2-1$,
\begin{equation}\label{equ:4.27}
\begin{split}
\phi(t)^2-|x-\nu_j|^2\leq C \Big(\big(\phi(t)+M\big)^2-|x|^2\Big).
\end{split}
\end{equation}
On the other hand, if $2\phi(t)< M^2-1$, then
\begin{equation}\label{equ:4.28}
\begin{split}
\phi(t)^2-|x-\nu_j|^2\leq\phi(t)^2\leq C_M\leq C_M\Big(\big(\phi(t)+M\big)^2-|x|^2\Big).
\end{split}
\end{equation}
Therefore,
\begin{equation*}
\begin{split}
&\Big\|\Big(\big(\phi(t)+M\big)^2-|x|^2\Big)^{\gamma_1}w\Big\|_{L^q([\frac{T_0}{2},\infty)\times\mathbb{R})}\\
&\leq C\sum_{j=1}^{N_0}\Big\|\Big(\big(\phi(t)+M\big)^2-|x|^2\Big)^{\gamma_1}w_j\Big\|_{L^q(Q_j)} \\
&\leq C\sum_{j=1}^{N_0}\big\|\big(\phi(t)^2-|x-\nu_j|^2\big)^{\gamma_1}w_j\big\|_{L^q(Q_j)}\qquad\qquad \text{\big(by \eqref{equ:4.23}\big)} \\
&\leq C\sum_{j=1}^{N_0}\big\|\big(\phi(t)^2-|x-\nu_j|^2\big)^{\gamma_2}F_j\big\|_{L^{\frac{q}{q-1}}(Q_j)}\qquad\quad \text{\big(by \eqref{equ:4.14}\big)}  \\
&\leq C\sum_{j=1}^{N_0}\Big\|\Big(\big(\phi(t)+M\big)^2-|x|^2\Big)^{\gamma_2}F_j\Big\|_{L^{\frac{q}{q-1}}(Q_j)} \qquad \text{\big(by \eqref{equ:4.28}\big)} \\
&\leq C_{N_0}\Big\|\Big(\big(\phi(t)+M)^2-|x|^2\big)^{\gamma_2}F\Big\|_{L^{\frac{q}{q-1}}([\frac{T_0}{2},\infty)\times\mathbb{R})},
\end{split}
\end{equation*}
which derives \eqref{equ:4.13}.
\end{proof}

\section{Proof of Theorem~\ref{thm:1.2}}
To establish the global existence, we shall choose \(q=p+1\) in Theorem 2.1 and Theorem 3.3, and
subsequently consider the following two cases:

\vskip 0.2 true cm

\textbf{Case I.  \(p_{\crit}<p<p_0=9\)}

\vskip 0.2 true cm

By the local existence and regularity of weak solution $u$ to (1.2)
(for examples, one can see \cite{Rua1} or references therein), one knows that $u\in C^{\infty}([0, T_0]\times\Bbb R)$
exists for any fixed constant $T_0<1$ and $u$ has a compact support on the variable $x$. Moreover, for any $N\in\Bbb N$,
\begin{equation}
\begin{split}
&\Big\|{u}\Big(\f{T_0}{2}, \cdot\Big)\Big\|_{C^N}+\Big\| \partial_t{u}\Big(\f{T_0}{2}, \cdot\Big)\Big\|_{C^N}\le C_{N}\ve.\\
\end{split}
\label{equ:6.1}
\end{equation}
Then we can take $\Big({u}\big(\f{T_0}{2}, x\big), \partial_t{u}\big(\f{T_0}{2}, x\big)\Big)$ as the new initial data to
solve (1.2) from $t=\f{T_0}{2}$.

Now we use the standard Picard iteration to prove Theorem 1.1. Let $u_{-1}\equiv0$, and for $k=0,1,2,3,\ldots$, let $u_k$ be the solution of
the following equation
\begin{equation*}%\label{equ:6.2}
\begin{cases}
&\partial_t^2 u_k-t\partial_x^2 u_k=F_p(t,u_{k-1}), \quad (t,x)\in\big(\f{T_0}{2}, \infty\big)\times\mathbb{R},\\
&u_k\big(\f{T_0}{2},x\big)={u}\big(\f{T_0}{2},x\big)\quad \partial_tu_k\big(\f{T_0}{2},x\big)=\partial_t{u}\big(\f{T_0}{2}, x\big).
\end{cases}
\end{equation*}
For  $p>p_{crit}$, we can fix a number $\gamma$ satisfying
\begin{equation*}\label{equ:6.1.1}
  \gamma<\frac{1}{p(p+1)},
\end{equation*}
\begin{equation*}\label{equ:6.1.2}
  \gamma<\f{1}{6}-\f{5}{6(p+1)},
\end{equation*}
and
\begin{equation*}\label{equ:6.1.3}
  \left(p-1\right)\gamma+\f{1}{6}>\f{5}{3(p+1)}.
\end{equation*}
Set
\begin{align*}
M_k=&\Big\|\Big(\big(\phi(t)+M\big)^2-|x|^2\Big)^\gamma u_k\Big\|_{L^q([\f{T_0}{2},\infty)\times\mathbb{R})}, \\
N_k=&\Big\|\Big(\big(\phi(t)+M\big)^2-|x|^2\Big)^\gamma (u_k-u_{k-1})\Big\|_{L^q([\f{T_0}{2},\infty)\times\mathbb{R})},
\end{align*}
where $q=p+1$. By \eqref{equ:6.1} and Theorem 2.1, we know that there exists a constant $C_0>0$ such that
\[M_0\leq C_0\ve.\]
Notice that for $j$, $k\geq0$,
\begin{equation*}
\begin{cases}
&\partial_t^2 (u_{k+1}-u_{j+1})-t \partial_x^2(u_{k+1}-u_{j+1}) =V(u_k,u_j)(u_k-u_j),   \\
&(u_{k+1}-u_{j+1})\big(\f{T_0}{2}, x\big)=0, \quad \partial_t(u_{k+1}-u_{j+1})\big(\f{T_0}{2}, x\big)=0,
\end{cases}
\end{equation*}
where
\begin{equation*}%\label{equ:6.3}
\big|V(u_k,u_j)\big|\le
\begin{cases}
&C(|u_k|+|u_j|)^{p-1} \quad \text{if} \quad t\geq T_0, \\
&C(1+|u_k|+|u_j|)^{p-1} \quad \text{if} \quad \f{T_0}{2}\leq t\leq T_0.
\end{cases}
\end{equation*}
%By our assumptions
%\[0<\gamma<\f{1}{6}-\f{1}{p+1}\quad \text{and}\quad \left(1-\f{1}{p}\right)(-\gamma)+\f{1}{6}+\frac{1}{3(p+1)}\geq\f{2}{p+1},\]
Then applying Theorem 3.2 and H\"{o}lder's inequality yields that for \(q=p+1\),
\begin{equation}\label{equ:6.4}
\begin{split}
&\Big\|\Big(\big(\phi(t)+M\big)^2-|x|^2\Big)^\gamma (u_{k+1}-u_{j+1})\Big\|_{L^q([\f{T_0}{2},\infty)\times\mathbb{R})} \\
&\leq C\Big\|\Big(\big(\phi(t)+M\big)^2-|x|^2\Big)^{p\gamma}V(u_k,u_j)(u_k-u_j)\Big\|_{L^{\frac{q}{q-1}}([\f{T_0}{2},\infty)\times\mathbb{R})} \\
&\leq C\Big\{\Big\|\Big(\big(\phi(t)+M\big)^2-|x|^2\Big)^\gamma (1+|u_k|+|u_j|)\Big\|_{L^q([\f{T_0}{2}, T_0]\times\mathbb{R})} \\
&\qquad +\Big\|\Big(\big(\phi(t)+M\big)^2-|x|^2\Big)^\gamma (|u_k|+|u_j|)\Big\|_{L^q([T_0,\infty]\times\mathbb{R})}\Big\}^{p-1} \\
&\quad \times \Big\|\Big(\big(\phi(t)+M\big)^2-|x|^2\Big)^\gamma (u_k-u_j)\Big\|_{L^q([\f{T_0}{2},\infty)\times\mathbb{R})} \\
&\leq C\big(C_1T_0^{\frac{1}{q}}+M_k
+M_j\big)^{p-1}\Big\|\Big(\big(\phi(t)+M\big)^2-|x|^2\Big)^\gamma (u_k-u_j)\Big\|_{L^q([\f{T_0}{2},\infty)\times\mathbb{R})}.
\end{split}
\end{equation}
If $j=-1$, then $M_j=0$, and we conclude that from \eqref{equ:6.4}
\[M_{k+1}\leq M_0+\frac{M_k}{2}\quad \text{for} \quad C\big(C_1T_0^{\frac{1}{q}}+M_k\big)^{p-1}\leq\frac{1}{2}.\]
This yields that
\[M_k\leq 2M_0 \quad \text{if} \quad C\big(C_1T_0^{\frac{1}{q}}+C_0\ve\big)^{p-1}\leq\frac{1}{2}.\]
Thus we get the boundedness of $\{u_k\}$ in the space $L^q(\Bbb R_+^{1+1})$ when the fixed constant $T_0$
and $\ve>0$ are sufficiently small. Similarly, we have
\[N_{k+1}\leq\frac{1}{2}N_k,\]
which derives that there exists a function $u\in L^q\big(\big[\f{T_0}{2}, \infty\big)\times\mathbb{R}\big)$ such that  $u_k\rightarrow u\in L^q\big(\big[\f{T_0}{2}, \infty\big)\times\mathbb{R}\big)$.
In addition, by the uniform boundedness of $M_k$ and the computations above, one easily obtains
\begin{align*}
&\parallel F_p(t,u_{k+1})-F_p(t,u_k)\parallel_{L^{\frac{q}{q-1}}([\f{T_0}{2}, \infty)\times\mathbb{R})} \\
&\leq C\parallel u_{k+1}-u_k\parallel_{L^q([\f{T_0}{2}, \infty)\times\mathbb{R})} \\
&\leq C\phi\Big(\frac{T_0}{4}\Big)^{-\gamma}N_k \\
&\le C2^{-k}.
\end{align*}
Therefore $F_p(t,u_k)\rightarrow F_p(t,u)$ in $L^{\frac{q}{q-1}}\big(\big[\f{T_0}{2}, \infty\big)\times\mathbb{R}\big)$ and hence $u$ is a weak solution of (1.2)
in the sense of distributions.

\vskip 0.2 true cm

\textbf{Case II. \(p\geq p_0=9\)}

\vskip 0.2 true cm

In this case, Theorem 1.1 can be achieved by completely analogous method in Theorem 1.2 of \cite{HWYin1}, we omit it here.

Combining {\bf Case I} and {\bf Case II}, we complete the proof of Theorem 1.1.

\section{Proof of Theorem 1.2.}
In this section, we shall show Theorem~\ref{thm1.1}.
 Motivated by \cite{Yor}, we introduce
the function $G(t)=\int_{\R}u(t,x)\,\md x$. By some estimates from \cite{HWYin1},
we can obtain a Riccati-type differential inequality for $G(t)$ so
that blowup of $G(t)$ is deduced from the following result (see Lemma 4 of
\cite{Sid}):
\begin{lemma}\label{lem2.1}
Suppose that $G\in C^2([a,b);\R)$ and, for $a\leq t<b$,
\begin{align}
G(t)&\geq C_0(R+t)^\alpha, \label{equ:2.1}\\
G''(t)&\geq C_1(R+t)^{-q}G(t)^p, \label{equ:2.2}
\end{align}
where $C_0$, $C_1$, and $R$ are some positive constants. Suppose further
that $p>1$, $\alpha\geq 1$, and $\left(p-1\right)\alpha\geq q-2$. Then
$b$ is finite.
\end{lemma}

In view of $\operatorname{supp} u_i\subseteq (-M, M)$ ($i=0,1$) and
the finite propagation speed of solutions to hyperbolic equations,
one has that, for any fixed $t>0$, the support of $u(t,\cdot)$ with
respect to the variable $x$ is contained in the interval $\big(-M-\phi(t), M+\phi(t)\big)$, where $\phi(t)=
\f{2}{3}\,t^{\f{3}{2}}$. Then it follows from an integration by
parts that
\begin{equation*}
G''(t)=\int_{\R}|u(t,x)|^p\, dx
\geq \frac{\left|\int_{\R}u(t,x)\,dx\right|^p}{
\left(\int_{|x|\leq M+\phi(t)}\,dx\right)^{p-1}}
\geq C(M+t)^{-\frac{3}{2}\,(p-1)}\,|G(t)|^p,
\end{equation*}
which means that $G(t)$ fulfills inequality \eqref{equ:2.2} with $
q=\frac{3}{2}\,\left(p-1\right)$ (once inequality \eqref{equ:2.1}
has been verified, we then know that $G$ is positive).

To establish
\eqref{equ:2.1}, we will introduce a suitable test function. The modified Bessel function
is
\begin{equation*}
K_\nu(t)=\int_0^\infty e^{-t\cosh{z}}\cosh(\nu z)dz, \quad \nu\in \R,
\end{equation*}
which satisfies
\[
\left(t^2\frac{d^2}{dt^2}+t\frac{d}{dt}-(t^2+\nu^2)\right)K_\nu(t)=0, \quad t>0.
\]

From page 24 of \cite{Erd2}, we have
\begin{equation}\label{equ:2.5}
  K_\nu(t)=\sqrt{\frac{\pi}{2t}}\,e^{-t}\left(1+O(t^{-1})\right) \quad
  \text{as $t\rightarrow \infty$,}
\end{equation}
provided that $\operatorname{Re}\nu>-1/2$. Set
\begin{equation}\label{equ:2.6}
\lambda(t)=C_1\,t^{\frac{1}{2}}K_{\frac{1}{3}}\left(\frac{2}{3}\,t^{\frac{3}{2}}\right),
\quad t>0,
\end{equation}
where the constant $C_1>0$ is chosen so that $\lambda(t)$ satisfies
\begin{equation}\label{equ:2.7}
\left\{ \enspace
\begin{aligned}
&\la''(t)-t\la(t)=0, && t\geq0 \\
&\la(0)=1, \quad \la(\infty)=0.
\end{aligned}
\right.
\end{equation}

Introduce the test function $\psi$ with
\begin{equation}\label{equ:2.9}
\psi(t,x)=\lambda(t)\varphi(x),
\end{equation}
where $\varphi=e^x$.
Let
\begin{equation}\label{equ:2.10}
G_1(t)=\int_{\R}u(t,x)\psi(t,x)\,dx.
\end{equation}
Then
\begin{equation}\label{equ:2.11}
G''(t)=\int_{\R}|u(t,x)|^p\, dx \geq
\frac{|G_1(t)|^p}{\left(\int_{|x|\leq
    M+\phi(t)}\psi(t,x)^{\frac{p}{p-1}}\,dx\right)^{p-1}}.
\end{equation}

Since the function \(\varphi(x)>0\) holds for all \(x\in\R\), one can repeat the proof of Lemma 2.3 in \cite{HWYin1} with little modification to get

\begin{lemma}\label{lem2.3}
Under the assumptions of\/ \textup{Theorem~\ref{thm1.1}}, there exists
a $t_0>0$ such that
\begin{equation}\label{equ:2.12}
G_1(t)\geq C\,t^{-\frac{1}{2}}, \quad t\geq t_0.
\end{equation}
\end{lemma}

Based on Lemma~\ref{lem2.3}, we are now able to prove
Theorem~\ref{thm1.1}.

\vskip 0.2 true cm

{\bf Proof of Theorem~\ref{thm1.1}.}
By \eqref{equ:2.5} and \eqref{equ:2.6}, we have that
\[
\lambda(t)\sim t^{-\frac{1}{4}}e^{-\phi(t)} \quad \text{as $t\rightarrow \infty$.}
\]
Next we estimate the denominator $\left(\int_{|x|\leq
  M+\phi(t)}\psi(t,x)^{\frac{p}{p-1}} \, dx\right)^{p-1}$ in
\eqref{equ:2.11}. Note that
\begin{equation*}
\left(\int_{|x|\leq
  M+\phi(t)}\psi(t,x)^{\frac{p}{p-1}}\,dx\right)^{p-1}
=\lambda(t)^p\left(\int_{|x|\leq
  M+\phi(t)}\varphi(x)^{\frac{p}{p-1}}\,dx\right)^{p-1}
\end{equation*}
and
\[
|\varphi(x)|\leq Ce^{|x|}.
\]
Then
\begin{align*}
\int_{|x|\leq M+\phi(t)}\varphi(x)^{\frac{p}{p-1}}\,dx & \le C\int_0e^{\frac{p}{p-1}r}\,dr
+C\int_{\frac{M+\phi(t)}{2}}^{M+\phi(t)}e^{\frac{p}{p-1}r}\,dr \\
& \leq C e^{\f{p}{p-1}\cdot\frac{M+\phi(t)}{2}}+e^{\f{p}{p-1}\left(M+\phi(t)\right)} \\
& \leq Ce^{\f{p}{p-1}\left(M+\phi(t)\right)}
\end{align*}
and
\begin{equation}\label{equ:2.16}
\begin{aligned}
\left(\int_{|x|\leq M+\phi(t)}\psi(t,x)^{\frac{p}{p-1}}\,dx\right)^{p-1}
&\leq Ct^{-\frac{p}{4}}e^{-p\phi(t)}e^{p\left(M+\phi(t)\right)} \\
&\leq Ct^{-\frac{p}{4}}.
\end{aligned}
\end{equation}
Therefore, it follows from \eqref{equ:2.11} and \eqref{equ:2.16} that, for $t\ge t_0$,
\begin{equation}\label{equ:2.17}
G''(t)\geq ct^{-\frac{p}{4}}.
\end{equation}

If \(-\f{p}{4}>-1\), then \(p<4\). In this case, we have
\[G(t)\geq\left(M+t\right)^{2-\f{p}{4}}\]
and further get \(\alpha=2-\f{p}{4}\) in \eqref{equ:2.1}. Hence the condition $\left(p-1\right)\alpha\geq q-2$ is written as
\[(p-1)\left(2-\f{p}{4}\right)>\f{3}{2}(p-1)-2,\]
which is equivalent to
\begin{equation}\label{equ:2.18}
  p^2-3p-6<0
\end{equation}
The positive root of \(p^2-3p-6=0\) is
\[p_1=\f{3+\sqrt{33}}{2}.\]
It follows from \eqref{equ:2.18} that \(p<p_1\). Direct computation shows \(p_1>4\). Thus, by Lemma 5.1 we know that
the  weak solution $u$ of problem \eqref{equ:original} will blow up in finite time when \(1<p<4\).

%If \(-\f{mp}{4}=-1\) or \(1<p=\f{4}{m}\), then we have
%\[G(t)\geq(M+t)\ln{(M+t)}\geq C(M+t),\]
%by an analysis similar to that in \cite{HWYin3}, we can get the blowup result for \(p=\f{4}{m}\), \(m=1, 2, 3\).

While if \(-\f{p}{4}\leq-1\) or \(p\geq 4\), we have
\[G(t)\geq C(M+t),\]
then we can take \(\alpha=1\) in \eqref{equ:2.1} and further obtain
\[(p-1)>\f{3}{2}(p-1)-2\Longrightarrow p<p_2=5,\]
which means that the solution $u$ of problem \eqref{equ:original} will blow up  for \(4\le p<5\)
in terms of Lemma 5.1.

Collecting these results above, we complete the proof of Theorem~\ref{thm1.1}.
\qquad \qquad \qquad \qquad \quad $\square$

%\begin{remark}
  %If one check the proof of Lemma~\ref{lem2.1} in \cite{Sid}, one shall see that Lemma~\ref{lem2.1} can only be applied to the function \(G(t)\) which satisfies
  %\[G''(t)\sim G(t)(1+t)^{-2}.\]
  %This is the reason we choose \(\alpha=2-\f{p}{4}\) in the case \(p\geq 4\).
%\end{remark}

\end{document}